\documentclass[11pt]{amsart}
\usepackage{amsmath, amsthm, mathabx,amscd, amsfonts, amssymb, hyperref, mathrsfs, textcomp, epsfig, a4wide, graphicx, IEEEtrantools, tikz, verbatim, xypic }
\usepackage[overload]{empheq}

\usepackage[all,cmtip]{xy}
\usetikzlibrary{matrix, decorations.pathreplacing,angles,quotes,cd}

\newtheorem{question}{Question}
\newtheorem{prop}{Proposition}[section]
\newtheorem{thm}[prop]{Theorem}
\newtheorem{lemma}[prop]{Lemma}
\newtheorem{cor}[prop]{Corollary}
\theoremstyle{definition}
\newtheorem{defn}[prop]{Definition}

\usepackage{mathtools}

\DeclarePairedDelimiter\floor{\lfloor}{\rfloor}

\newcommand{\Pin}{\mathrm{Pin}(2)}

\newcommand{\spin}{\mathfrak{s}}
\newcommand{\EHess}{\widehat{\mathrm{Hess}}}

     \RequirePackage{rotating}                   
    \def\HMt{%
       \setbox0=\hbox{$\widehat{\mathit{HM}}$}
       \setbox1=\hbox{$\mathit{HM}$}
       \dimen0=1.1\ht0
       \advance\dimen0 by 1.17\ht1
       \smash{\mskip2mu\raise\dimen0\rlap{%
          \begin{turn}{180}
              {$\widehat{\phantom{\mathit{HM}}}$}
           \end{turn}} \mskip-2mu    
                \mathit{HM}
    }{\vphantom{\widehat{\mathit{HM}}}}{}}
    
         \RequirePackage{rotating}                   
    \def\HSt{%
       \setbox0=\hbox{$\widehat{\mathit{HS}}$}
       \setbox1=\hbox{$\mathit{HS}$}
       \dimen0=1.1\ht0
       \advance\dimen0 by 1.17\ht1
       \smash{\mskip2mu\raise\dimen0\rlap{%
          \begin{turn}{180}
              {$\widehat{\phantom{\mathit{HS}}}$}
           \end{turn}} \mskip-2mu    
                \mathit{HS}
    }{\vphantom{\widehat{\mathit{HS}}}}{}}

    \newcommand{\HMb}{\overline{\mathit{HM}}}
\newcommand{\HMf}{\widehat{\mathit{HM}}}
    \newcommand{\HSb}{\overline{\mathit{HS}}}
\newcommand{\HSf}{\widehat{\mathit{HS}}}

\newcommand{\Rin}{\mathcal{R}}
\newcommand{\V}{\mathcal{V}}
\newcommand{\dd}{\mathbf{d}}

\newcommand{\ztwo}{\mathbb{F}}

\theoremstyle{remark}
\newtheorem{remark}{Remark}[section]

\begin{document}
\title{Homology cobordism and the geometry of hyperbolic three-manifolds}

\author{Francesco Lin}
\address{Department of Mathematics, Columbia University} 
\email{flin@math.columbia.edu}

\begin{abstract}
A major challenge in the study of the structure of the three-dimensional homology cobordism group is to understand the interaction between hyperbolic geometry and homology cobordism. In this paper, for a hyperbolic homology sphere $Y$ we derive explicit bounds on the relative grading between irreducible solutions to the Seiberg-Witten equations and the reducible one in terms of the spectral and Riemannian geometry of $Y$. Using this, we provide explicit bounds on some numerical invariants arising in monopole Floer homology (and its $\Pin$-equivariant refinement). We apply this to study the subgroups of the homology cobordism group generated by hyperbolic homology spheres satisfying certain natural geometric constraints.

\end{abstract}

\maketitle

\section*{Introduction}
Differential topology in $4=3+1$ dimensions is dramatically different from all other dimensions. An object that perfectly encapsulates this radical difference is the homology cobordism group $\Theta_n$ consisting of smooth homology $n$-spheres up to homology cobordism with connected sum as operation: it was shown by Kervaire \cite{Ker} that $\Theta_n$ is always a finite group for $n\neq3$, while Furuta \cite{Fur} proved that $\Theta_3$ contains a subgroup isomorphic to $\mathbb{Z}^{\infty}$ (hence in particular it is not finitely generated).
\par
In the past ten years, the study of the algebraic structure of $\Theta_3$ has attracted considerable renewed attention following Manolescu's disproof of the longstanding Triangulation Conjecture in high dimensions \cite{Man}: work of Galewski-Stern \cite{GS} and Matumoto \cite{Mat} shows that this is implied by the Rokhlin homomorphism
\begin{equation*}
\mu:\Theta_3\rightarrow \mathbb{Z}/2\mathbb{Z}
\end{equation*}
not being split. Manolescu proved the latter by exploiting $\Pin$-symmetry in Seiberg-Witten theory. Since then, related ideas in the context of Heegaard Floer homology \cite{HM} have been used to show that $\Theta_3$ admits a $\mathbb{Z}^{\infty}$ summand \cite{DHST}. Despite this, still very little is known about the algebraic structure of $\Theta_3$: for example, it is an open question whether there are torsion or divisible elements in $\Theta_3$.
\\
\par
As we are dealing with three-manifolds, it is natural to approach this question from the point of view of Thurston's geometrization. Using tools from involutive Heegaard Floer homology, it was shown in \cite{HHSZ} that Seifert fibered spaces do not generate $\Theta_3$, and in fact the quotient of $\Theta_3$ by Seifert spaces is infinitely generated \cite{HHSZ1}. In particular, Seifert manifolds are too narrow of a class to understand the general properties of $\Theta_3$.
\par
On the other hand, Myers showed \cite{Mye} that every element in $\Theta_3$ admits a hyperbolic representative. Because of this, in order to understand $\Theta_3$, a key problem is to understand the homology cobordism properties of hyperbolic three-manifolds. By Mostow rigidity, geometric invariants of the hyperbolic metric are topological invariants, and we are led to the following fundamental question.
\begin{question}\label{quest1}
For a hyperbolic homology sphere $Y$, is there any relation between the hyperbolic invariants of $Y$ (e.g. volume, injectivity radius, etc.) and the properties of the corresponding class $[Y]\in\Theta_3$?
\end{question}

As a more specific instance of this question, we can consider for $V,\varepsilon>0$ the subgroup $\Theta_{V,\varepsilon}\subset \Theta_3$ generated by the integral homology spheres with volume $\leq V$ and injectivity radius $\geq\varepsilon$. As the latter is a finite set \cite{BP}, $\Theta_{V,\varepsilon}$ is a finitely generated hence proper subgroup of $\Theta_3$. Furthermore, by Myers' theorem we have that
\begin{equation*}
\bigcup_{V,\varepsilon>0}\Theta_{V,\varepsilon}=\Theta_3,
\end{equation*}
hence this family of subgroups provides a natural geometric filtration on $\Theta_3$ by finitely generated subgroups.
\par
A very basic question is then the following.
\begin{question}\label{quest2}
Given $V,\varepsilon>0$, can one exhibit an \textbf{explicit} integral homology sphere $Y$ for which $[Y]\not\in\Theta_{V,\varepsilon}$?
\end{question}

A first complication in addressing this question is that that we do not have a nice `concrete' description of the generators of $\Theta_{V,\varepsilon}$. For example, the well-known Hodgson-Weeks census \cite{HW} is only an approximation of the set of closed orientable hyperbolic manifolds (not necessarily homology spheres) with volume $\leq6.5$ and injectivity radius $\geq0.15$. Furthermore, even if we had such description, understanding the homology cobordism properties of these spaces and their connected sums is by far the most challenging part of problem.
\\
\par
In this paper, while we do not address Question \ref{quest2} directly, we study the analogous of Question \ref{quest2} for certain subgroups of $\Theta_3$ whose definition also involves an interesting quantity arising in spectral geometry, namely the first eigenvalue of the Hodge Laplacian on coexact $1$-forms $\lambda_1^*(Y)$ (for the hyperbolic metric). Our main theorem will then essentially reduce Question \ref{quest2} to a natural but very subtle question in the spectral geometry of hyperbolic homology spheres.
\begin{defn}
Fix $V,\varepsilon,\delta>0$. We define $\Theta_{V,\varepsilon,\delta}$ to be the subgroup of $\Theta_3$ generated by homology spheres $Y$ such that the volume $\mathrm{vol}(Y)$ of $Y$ is at most $V$, the injectivity radius $\mathrm{inj}(Y)$ (or equivalently, half the systole) of $Y$ is at least $\varepsilon$, and the first eigenvalue of the Hodge Laplacian on coexact $1$-forms $\lambda_1^*(Y)$ is at least $\delta$.
\end{defn}
The quantity $\lambda_1^*(Y)$, which is spectral in nature, can be studied in terms in geometric quantities of $Y$, namely its volume and the complex lengths of its geodesics, using a version of the Selberg trace formula \cite{LL}. For example, among the small volume manifolds in the Hodgson-Weeks census, it is shown in \cite{LL} that:
\begin{itemize} 
\item for the manifold labeled $5$ in the Hodgson-Weeks census, which is a homology sphere with volume $\approx 1.39$ and injectivity radius $\approx 0.18$,  $\lambda_1^*>0.04$.
\item for the manifold labeled $34$ in the Hodgson-Weeks census, which is a homology sphere with volume $\approx 1.91$ and injectivity radius $\approx 0.24$, and $\lambda_1^*>0.001$.
\end{itemize}
\begin{remark}Heuristically, one expects in general $\lambda^*_1$ of a hyperbolic integral homology sphere to be quite small. Indeed, any example with $\lambda^*_1>2$ would be an $L$-space by \cite{LL} hence a counterexample to a conjecture of Ozsv\'ath and Szab\'o that the only irreducible integral homology spheres which are $L$-spaces are $S^3$ and the Poincar\'e homology sphere.
\end{remark}
Notice that again that the groups $\Theta_{V,\varepsilon,\delta}$ are finitely generated and Myers' theorem implies that 
\begin{equation*}
\bigcup_{V,\varepsilon,\delta>0}\Theta_{V,\varepsilon,\delta}=\Theta_3.
\end{equation*}
hence also provide a natural geometric filtration by finitely generated subgroups.
\begin{remark}
Of course, we are mostly interested in the case in which $V$ is large, and $\varepsilon$ and $\delta$ are small. To simplify some of the statements, we will assume throughout the paper for convenience that $V\geq0.94$ (this is true for all hyperbolic manifolds by \cite{Mil}), $\varepsilon\leq0.15$ and $\delta\leq 1$.
\end{remark}
We are now ready to state our main result, which answers the analogue of Question \ref{quest2} for the subgroups $\Theta_{V,\varepsilon,\delta}$.
\begin{thm}\label{mainthm}
Given $V,\varepsilon,\delta>0$, there exists an \textbf{explicitly computable} constant $\mathfrak{n}_{V,\varepsilon,\delta}$ such that for $n> \mathfrak{n}_{V,\varepsilon,\delta}$ the Brieskorn sphere $Y_n=\Sigma(2,8n-1,16n-1)$ does not belong to $\Theta_{V,\varepsilon,\delta}$. For example when $\varepsilon=0.15$, the constant
\begin{equation*}
\mathfrak{n}_{V,0.15,\delta}=  4V \cdot\left[200+ \exp\left(11+15\cdot e^{11/2}\cdot V^{7/12}\cdot\left(\cosh(57V)-1\right)^{8/3}\cdot \left(1+\frac{3}{\delta}\right)^{1/2}\right)\right]+6
\end{equation*}
works.
\end{thm}
More in general, our proof provides a completely explicit formula $\mathfrak{n}_{V,\varepsilon,\delta}$ in terms of $V,\varepsilon,\delta$ only involving elementary operations; the expression itself is quite long and the most unpleasant dependence is that on $\varepsilon$, which is the reason why we specialized to $\varepsilon=0.15$ in the statement of the theorem. Of course, the constant $\mathfrak{n}_{V,\varepsilon,\delta}$ that we provide is very far from being sharp (even within the limitations of our techniques).
\\
\par
Let us comment on the role of the three quantities $\mathrm{vol}(Y), \mathrm{inj}(Y)$ and $\lambda_1^*(Y)$ that appear in the statement of the theorem. As mentioned above, there are only finitely many hyperbolic three-manifolds with $\mathrm{vol}(Y)\leq V$ and $\mathrm{inj}\geq\varepsilon$. Of course, this finiteness properties implies that there is a lower bound on $\lambda_1^*(Y)$ in such set of generators in terms of $V,\varepsilon$. The problem is that there is no known \textit{explicit} bound in terms of those quantities. The bounds provided in \cite{LL} take as input significantly more information, namely the volume and a sufficiently large portion of the complex length spectrum of $Y$; furthermore, they involve an optimization procedure that makes them implicit in nature.
\begin{remark}
Let us also point out that the role of $\lambda_1^*$ in Theorem \ref{mainthm} is quite different from the one it played in \cite{LL}. There the lower bound $\lambda_1^*>2$ is used to rule out the \textit{existence} of irreducible solutions to the Seiberg-Witten equations. In the present paper the assumption $\lambda_1^*\geq\delta$ will be used instead to provide \textit{a priori bounds} on their norm, and our arguments will not tell us anything about the number of solutions (which is possibly zero in certain cases).
\end{remark}
In fact, the study of lower bounds for $\lambda_1^*(Y)$ for hyperbolic three-manifolds (not necessarily rational homology spheres) has recently attracted a lot of attention in light of its role in a conjecture on Bergeron and Venkatesh on the growth of torsion in the homology of congruence covers of a fixed arithmetic manifold \cite{BV}. Interesting work providing lower bounds in terms of geometric quantities describing the areas of surfaces with boundary given curves can be found in \cite{LS}, \cite{Rud} and \cite{BC}. Lower bounds of a different nature were provided in \cite{Jam}.
\par
Notice that for hyperbolic rational homology spheres $Y$, a lower bound on $\lambda_1^*(Y)$ implies a lower bound on $\mathrm{inj(Y)}$ which can be \textit{explicitly} computed (see \cite[Section $5$]{LL}). In particular, an analogue of Theorem \ref{mainthm} can be stated for the subgroup $\Theta_{V,\delta}$ generated by hyperbolic homology spheres  $Y$with $\mathrm{vol}(Y)\leq V$ and $\lambda_1^*(Y)\geq \delta$. We decided to keep the dependence on $\mathrm{inj(Y)}$ in the statement of Theorem \ref{mainthm} because the injectivity radius is a very natural quantity to study and plays an explicit role in several points in our argument.
\par
Finally, even though we have focused our discussion on the spectral gap $\lambda_1^*$ (which is by far the hardest quantity to understand), the Riemannian geometry of homology spheres is very challenging to study itself. For example, it is not known whether there are homology spheres with arbitrarily large injectivity radius; for results in this direction, see \cite{BD}.

\begin{remark}
As it will be apparent from the proof, instead of a lower bound on the injectivity radius one might alternatively assume an upper bound on the diameter. This is very natural from the point of view of hyperbolic geometry, as very short geodesics have very large tubes around them, see also Theorem \ref{diamineq} below.
\end{remark}

\vspace{0.3cm}

\textbf{Strategy of the proof. }To relate hyperbolic geometry and homology cobordism we exploit as intermediary the monopole Floer homology package, and in particular the invariants arising from $\Pin$-symmetry \cite{Lin}. Roughly speaking, the proof of Theorem \ref{mainthm} consists of two distinct parts:
\begin{enumerate}
\item using a variety of tools from geometric analysis and spectral theory, we show that for the geometric generators of $\Theta_{V,\varepsilon,\delta}$ one can provide explicit upper bounds for the grading of any irreducible solution relative to the reducible one. This uses the interpretation via Atiyah-Patodi-Singer index theory \cite{APS3} of the relative grading in terms of the spectral flow of a family of operators, which can be bounded by making explicit certain quantities that arise in the classical proof of compactness of the moduli space of solutions to the Seiberg-Witten equations.
\item From this bound, one obtains restrictions on the monopole Floer homology of the generators of $\Theta_{V,\varepsilon,\delta}$. Using this, one obtains upper bounds on certain numerical homology cobordism invariants arising in $\Pin$-monopole Floer homology for all elements $\Theta_{V,\varepsilon,\delta}$. One then concludes because for the manifolds $Y_n$ such invariants are explicitly computed, and tend to infinity.
\end{enumerate}
In particular, there are many other families that we could choose for the statement of Theorem \ref{mainthm} instead of  $Y_n$.

\vspace{0.3cm}

\textbf{Other coefficients.}
Part (1) of the strategy outlined above is valid more generally for any rational homology sphere, so we can use this approach to study homology cobordism with different coefficients.
\par
Notice that because we are using $\Pin$-monopole Floer homology, the exact same strategy outlined in part (2) above allows to prove the analogue of Theorem \ref{mainthm} (with the same constant $\mathfrak{n}_{V,\varepsilon,\delta}$) for the group $\Theta^{\mathbb{Z}/2\mathbb{Z}}_3$ of $\mathbb{Z}/2\mathbb{Z}$-homology spheres up to $\mathbb{Z}/2\mathbb{Z}$-homology cobordism.
\par
Furthermore, combining part (1) with the explicit estimates for the degree of the reducible solutions in \cite{LL2} (under the assumption $\mathrm{vol}(Y)\leq V$ and $\mathrm{inj}\geq\varepsilon$) allows us to provide explicit upper bounds on the Fr\o yshov invariants $h$ of rational homology spheres $Y$ in terms of bounds on $\mathrm{vol}(Y), \mathrm{inj}(Y)$ and $\lambda_1^*(Y)$. More precisely, we have the following.
\begin{thm}\label{secondthm}
Given $V,\varepsilon,\delta>0$, there exists an \textbf{explicitly computable} constant $\mathfrak{m}_{V,\varepsilon,\delta}$ such that for all hyperbolic rational homology spheres $Y$ with $\mathrm{vol}(Y)\leq V$, $\mathrm{inj}(Y)\geq\varepsilon$, and $\lambda_1^*(Y)\geq\delta$, the inequality
\begin{equation*}
|h(Y,\spin)|\leq \mathfrak{m}_{V,\varepsilon,\delta}
\end{equation*}
holds for all spin$^c$ structures $\spin$ on $Y$.
\end{thm}
Using this, we can address the analogue of Question \ref{quest1} for the rational homology cobordism group $\Theta^{\mathbb{Q}}_3$ in terms of extra spectral geometry data. For example, if $N> \mathfrak{m}_{V,\varepsilon,\delta}$, then the connected sum of $N$ copies of the Poincar\'e homology sphere, which has $|h|=N$, is not rationally homology cobordant to any rational homology sphere $Y$ with $\mathrm{vol}(Y)\leq V$, $\mathrm{inj}(Y)\geq\varepsilon$, and $\lambda_1^*(Y)\geq\delta$.

\vspace{0.3cm}

\textbf{Structure of the paper. }In Section \ref{SWreview} we discuss the relevant background in Seiberg-Witten theory and monopole Floer homology needed for our purposes. In Section \ref{Pin2review} we discuss the homology cobordism invariants coming $\Pin$-monopole Floer homology that will be used to show that $Y_n$ does not belong to $\Theta_{V,\varepsilon,\delta}$. In Section \ref{sobolev} we review some results in the literature regarding the Sobolev embedding theorem relevant to our strategy. In Section \ref{explicitellip} we provide explicit bounds on a coexact $1$-form in terms of its exterior differential and the geometry of the underlying hyperbolic manifold. In Section \ref{functional}, we discuss some abstract operator theory that allows us to provide explicit bounds on the spectral flow between the reducible solution and the irreducible ones. In Section \ref{spectralbound} we provide explicit bounds on the perturbation in our case of interest. Finally, in Section \ref{mainproof} we recall some results about spectral densities and put the various pieces together to prove our main results.

\vspace{0.3cm}

\textbf{Acknowledgements. }The author was partially supported by the Alfred P. Sloan Foundation and NSF grant DMS-2203498.

\vspace{0.5cm}
\section{Some background in Seiberg-Witten theory}\label{SWreview}
In this section we review some fundamental concepts in monopole Floer homology, namely the proof of compactness and how gradings in Floer chain complex can be interpreted in terms of spectral flow.
We refer the reader to \cite{KM} for the definitive reference, and \cite{Lin5} for a friendly introduction.
\\
\par
\textit{The compactness argument.} Fix a spin$^c$ structure $\spin$ on $Y$. Recall that the Seiberg-Witten equations for an element $(B,\Psi)$ in the configuration space $\mathcal{C}=\mathcal{C}(Y,\spin)$ are
\begin{align*}[left=\empheqlbrace]
\frac{1}{2}\ast F_{B^t}+\rho^{-1}(\Psi\Psi^*)_0&=0\\ \label{SWfirst}
D_B\Psi&=0
\end{align*}
where $B^t$ is the connection induced on $\det S$. The left hand side is $\mathrm{grad} \mathcal{L}$, where $\mathcal{L}$ is the Chern-Simons-Dirac functional.
\begin{remark}\label{trans}
Of course, one needs to add perturbations in order to achieve transversality of the equations; on the other hand, it is straightforward to check that the arguments throughout the paper hold after choosing a sufficiently small perturbation. Given this, to keep the paper readable, we will always work with the case of the unperturbed equations, and assume that they are transversely cut out. Our key estimates on the spectral flow (and spectral densities in general) will deliberately not be sharp hence readily checked to hold also after small perturbations of the equations.
\end{remark}
 We will always be interested in the case of $b_1(Y)=0$, so that $\spin$ is a torsion spin$^c$ structure. We can then write $B=B_{\bullet}+b$ for a reference smooth connection $B_{\bullet}$ with $B_{\bullet}^t$ flat and $b$ is an imaginary valued $1$-form. The system of equations can then be expanded as
\begin{align}[left=\empheqlbrace]
\ast db+\rho^{-1}(\Psi\Psi^*)_0&=0\\ \label{SWfirst}
D_{B_{\bullet}}\Psi+\rho(b)\Psi&=0.
\end{align}
After a gauge transformation, we can make any configuration $(B,\Psi)$ to lie in the Coulomb gauge with respect to the reducible solution $(B_{\bullet},0)$, i.e. we can assume $d^*b=0$.
\\
\par
Let us now recall from \cite[Chapter 5]{KM} how one obtains a priori $L_1^2$ bounds on $(B,\Psi)$ on solutions in Coulomb gauge to the equations which are only assumed to be $L^2_1$ to begin with (any such solution turns out to be smooth, as recalled below). Our goal later will be to make (parts of) the estimates \textit{explicit}. The steps are the following.
\begin{enumerate}
\item an argument involving the Weitzenb\"ock formula and integration by parts implies that for any solution we obtain an a priori $L^4$ bound on $\Psi$, hence an $L^2$ bound on $(\Psi\Psi^*)_0$, in terms of the geometry of $Y$.
\item By G\r{a}rding inequality, there exists a constant $C>0$ such that
\begin{equation*}
\| b\|_{L^2_1}\leq C\|* db\|_{L^2},
\end{equation*}
where we used that $\ast d$ has no kernel acting on coclosed forms as $b_1(Y)=0$. In particular, the first equation this implies then that we have $L^2_1$ bounds on $b$. By the Sobolev embedding theorem, in dimension three this gives us $L^6$ bounds, hence a fortiori $L^4$ bounds, on $b$.
\item We then get by H\"older inequality an $L^2$ bound on $\rho(b)\Psi$, hence again by elliptic regularity the second equation implies an $L^2_1$ bound on $\Psi$.
\end{enumerate}
Notice that while for some points of the argument it is clear how to make the constants explicits in terms of the geometry of $Y$, other parts are not as clear (e.g. the constant in the Sobolev embedding $L^2_1\hookrightarrow L^6$). From the $L^2_1$ bound we get directly the existence of a subsequence strongly converging in $L^2$; to obtain compactness of the moduli space in stronger topologies one needs an extra argument which also proves that any $L^2_1$ solution $(B,\Psi)$ in the Coulomb gauge above is automatically smooth. Finally an argument involving the Weitzenb\"ock formula implies the $L^{\infty}$ bound
\begin{equation}\label{apriori}
|\Psi|^2\leq \max\{0,-\inf s/2\},
\end{equation}
see \cite[Section $4.6$]{KM}, which is a significant improvement on the $L^4$ bound we started with.
\\
\par
\textit{Extended Hessians.} Suppose now $(B_0,\Psi_0)$ is a solution to the equations (with $B_0=B_{\bullet}+b_0$), and let us study the linearized equations, i.e. compute $\mathrm{Hess}\mathcal{L}_{(B_0,\Psi_0)}$ as an operator on of the tangent space to the configuration space
\begin{equation*}
T_{(B_0,\Psi_0)}\mathcal{C}=i\Omega^1\oplus \Gamma(S).
\end{equation*}
For $(b,\Psi)$ a tangent vector to the Coulomb slice,
this sends
\begin{equation*}
\begin{bmatrix}
{b}\\ {\Psi}
\end{bmatrix}
\mapsto
\begin{bmatrix}
\ast d{b}+\rho^{-1}({\Psi}\Psi_0^*+\Psi_0{\Psi}^*)_0\\
D_{B_{\bullet}}{\Psi}+\rho({b})\Psi_0+\rho(b_0){\Psi}
\end{bmatrix}.
\end{equation*}
In particular, the Hessian at the reducible solution $(B_{\bullet},0)$ is simply
\begin{equation*}
\mathrm{Hess}\mathcal{L}_{(B_{\bullet},0)}=\ast d\oplus D_{B_{\bullet}},
\end{equation*}
and the Hessian at another configuration is then obtained by a suitable perturbation:
\begin{equation*}
\mathrm{Hess}\mathcal{L}_{(B_0,\Psi_0)}=\mathrm{Hess}\mathcal{L}_{(B_{\bullet},0)}+
\begin{bmatrix}
0 & \rho^{-1}(\cdot\Psi_0^*+\Psi_0\cdot^*)_0\\
\rho(\cdot)\Psi_0 & \rho(b_0)\cdot
\end{bmatrix}.
\end{equation*}
These Hessians are not the right operators to study from the point of view of Floer theory: for example, because of gauge invariance they always have infinite dimensional kernel. This can be solved by taking into account the linearization of the gauge group action as follows. Consider the operator
\begin{align*}
\dd_{(B_0,\Psi_0)}&: i\Omega^0\rightarrow i\Omega^1\oplus \Gamma(S)\\
\xi&\mapsto(-d\xi,\xi\Psi_0),
\end{align*}
which is the linearization of the gauge group action at $(B_0,\Psi_0)$. Its formal adjoint is given by
\begin{align*}
\dd_{(B_0,\Psi_0)}^*&:  i\Omega^1\oplus \Gamma(S)\rightarrow i\Omega^0\\
(b,\Psi)&\mapsto -d^*b+i\mathrm{Re}\langle i\Psi_0,\Psi\rangle.
\end{align*}
The extended Hessian at $(B_0,\Psi_0)$ \cite[Section 12.3]{KM} is then the operator on $i\Omega^0\oplus T\mathcal{C}$ given in block matrix form as
\begin{equation*}
\EHess\mathcal{L}_{(B_0,\Psi_0)}=\begin{bmatrix}
0 & \dd_{(B_0,\Psi_0)}^* \\
\dd_{(B_0,\Psi_0)} & \mathrm{Hess}\mathcal{L}_{(B_0,\Psi_0)}
\end{bmatrix};
\end{equation*}
it extends to a bounded operator from $L^2_1$ to $L^2$. This operator is first-order elliptic self-adjoint hence diagonalizable with discrete spectrum unbounded in both directions. 
More explicitly, we have as operators on $i\Omega^0\oplus i\Omega^1\oplus \Gamma(S) $ that
\begin{equation*}
\EHess\mathcal{L}_{(B_{\bullet},0)}=\begin{bmatrix}
0 & -d^*&0 \\
-d&\ast d&0\\
0&0&D_{B_{\bullet}}
\end{bmatrix}.
\end{equation*}
At another configuration $(B_0,\Psi_0)$ we have that
\begin{equation*}
\EHess\mathcal{L}_{(B_0,\Psi_0)}=\EHess\mathcal{L}_{(B_{\bullet},0)}+\begin{bmatrix}
0 & 0 &i\mathrm{Re}\langle i\Psi_0,\cdot\rangle \\
0&0 & \rho^{-1}(\cdot\Psi_0^*+\Psi_0\cdot^*)_0\\
\cdot \Psi_0& \rho(\cdot)\Psi_0 & \rho(b_0)\cdot
\end{bmatrix}.
\end{equation*}
The perturbation is a multiplication operator; because we have an $L^{\infty}$ bound on $\Psi$, the corresponding multiplication map is bounded from $L^2$ to $L^2$. On the other hand, the argument we outlined above does not guarantee an $L^{\infty}$ bound on $b$, but merely an $L^6$ bound (to be made explicit); via the Sobolev embedding theorem, this implies that the corresponding multiplication operator is bounded from $L^2_1$ to $L^2$.

\begin{remark}
The usual proof of the compactness of the Seiberg-Witten moduli space in fact gives $L^{\infty}$ bound on $b$; the problem is that the stronger bound is hard to make explicit, as it involves estimates in $L^p$ Sobolev spaces whose constants are hard to directly relate to the geometry of $Y$.
\end{remark}

\textit{The Floer chain complex.} Recall that (under suitable transversality assumptions, cf. Remark \ref{trans}), including that the spectrum of $D_{B_{\bullet}}$ is simple and does not include zero) the Floer chain complex computing $\HMt_*(Y,\spin)$ for a rational homology sphere consists of:
\begin{itemize}
\item one generator for each gauge equivalence class of irreducible solutions;
\item the gauge equivalence class of the reducible solution $[B_{\bullet},0]$ contributes an infinite tower of generators, one for each eigenspace of $D_{B_{\bullet}}$ with positive eigenvalue; we will denote by $[\mathfrak{b}_i]$ for $i\geq 0$ the critical point corresponding to the $i$th positive eigenvalue. 
\end{itemize}
In our case of interest $\spin$ is torsion, and the chain complex admits a natural relative $\mathbb{Z}$-grading. We have that for the reducible generators
\begin{equation*}
\mathrm{gr}([\mathfrak{b}_i],[\mathfrak{b}_j])=2(i-j)\text{ for all }i,j\in\mathbb{N}.
\end{equation*}
To describe the other relative gradings, recall that for a family of suitable operators (such as the extended Hessians $\EHess\mathcal{L}$ that we consider presently), $\{L_t\}$ with $t\in[0,1]$ such that $0$ does not belong to the spectrum of $L_0,L_1$, the \textit{spectral flow} is the number of eigenvalues that goes from negative to positive (counted with signs); this in general depends on the choice of path and not just on the endpoints. In our situation, we consider the family of extended Hessians $\EHess\mathcal{L}$, and the spectral flow between $\EHess\mathcal{L}_{(B,\Psi)}$ and $\EHess\mathcal{L}_{(B',\Psi')}$ (which under the assumption of transversality do not have zero in the spectrum), computed via a path of extended Hessians $\EHess\mathcal{L}_{(B_t,\Psi_t)}$ for configurations $(B_t,\Psi_t)$, is independent of the choice of path (this is not true if $\spin$ is not torsion). We then have
\begin{equation}\label{sfred}
\mathrm{gr}([\mathfrak{b}_0],[B,\Psi])=\mathsf{sf}\left(\EHess\mathcal{L}_{(B_{\bullet},0)},\EHess\mathcal{L}_{(B,\Psi)}\right),
\end{equation}
see \cite[Section 14.4]{KM}, where by $\mathsf{sf}$ we will always implicitly mean the spectral flow along a family of extended Hessians $\{\EHess\mathcal{L}_{(B_t,\Psi_t)}\}$.
\begin{remark}
This is the Floer theoretic analogue of the fact in Morse theory that the expected dimension of the space of parametrized trajectories between two critical points is the difference of the indices, and is a consequence of Atiyah-Singer-Patodi index theory \cite{APS3}.
\end{remark}
In our situation, we can consider the path of configurations for $t\in[0,1]$ given by $(B_{\bullet}+tb_0, t\Psi_0)$ for which the family of extended Hessians
\begin{equation}\label{exthess}
\EHess\mathcal{L}_{(B_{\bullet}+tb_0, t\Psi_0)}=\EHess\mathcal{L}_{(B_{\bullet},0)}+t\begin{bmatrix}
0 & 0 &i\mathrm{Re}\langle i\Psi_0,\cdot\rangle \\
0&0 & \rho^{-1}(\cdot\Psi_0^*+\Psi_0\cdot^*)_0\\
\cdot \Psi_0& \rho(\cdot)\Psi_0 & \rho(b_0)\cdot
\end{bmatrix}
\end{equation}
is linear in $t$. A main goal of the paper will be to provide an explicit bounds for the spectral flow of the family $\EHess\mathcal{L}_{(B_{\bullet}+tb_0, t\Psi_0)}$.
\begin{remark} By elliptic regularity, the spectral flow is independent of the choice of Sobolev completion; we will consider the family as acting between $L^2_1$ to $L^2$ as this is the case directly accessible from the geometry of the three-manifold.
\end{remark}

\vspace{0.3cm}
\textit{Floer homology and spectral flow.} We conclude this section by discussing the significance of such bound on the spectral flow for the structure of Floer homology. We will for with coefficients in $\ztwo=\mathbb{Z}/2\mathbb{Z}$. Recall that for a spin$^c$ rational homology sphere $(Y,\spin)$ we have the isomorphism of $\ztwo[U]$-modules
\begin{equation*}
\HMt_{*}(Y,\spin)=\left(\ztwo[U^{-1},U]/U\ztwo[U]\right)\bigoplus\left(\oplus\ztwo[U]/U^{n_i}\right)
\end{equation*}
for some collection of $n_i\geq 1$. The first summand is commonly referred to as the $U$-tower, and is the image of $\HMb_{*}(Y,\spin)$ under $i_*$.
\begin{defn}In the situation above, the $U$-\textit{torsion width} of $(Y,\spin)$, denoted by $t(Y,\spin)$ is defined as the maximum of the $n_i$.
\end{defn}

\begin{prop}\label{torsbound}
The torsion number $t(Y,\spin)$ satisfies the inequality
\begin{equation*}
t(Y,\spin)\leq2\max|\mathsf{sf}\left(\EHess\mathcal{L}_{(B_{\bullet},0)},\EHess_{\mathcal{L}(B,\Psi)}\right)|+2
\end{equation*}
where the maximum is taken over all irreducible solutions $[B,\Psi]$ of the Seiberg-Witten equations.
\end{prop}

\begin{proof}
For a summand $\ztwo[U]/U^{n_i}$ consider its top and bottom elements; these have a difference in grading of $2n_i-2$. Denote by $x_{\pm}$ cycles representing these two elements; these are not represented by reducible generators, as otherwise they would belong to the $U$-tower (cf. \cite[Section 22.2]{KM}). Then
\begin{equation*}
2n_i-2=\mathrm{gr}(x_+,x_-)=\mathrm{gr}(x_+,[\mathfrak{b}_0])+\mathrm{gr}([\mathfrak{b}_0],x_-)\leq|\mathrm{gr}(x_+,[\mathfrak{b}_0])|+|\mathrm{gr}([\mathfrak{b}_0],x_-)|
\end{equation*}
and the result follows from (\ref{sfred}).
\end{proof}

Finally, recall that because the spin$^c$ structure is torsion, there is a natural absolute $\mathbb{Q}$-grading $\mathrm{gr}^{\mathbb{Q}}$ on Floer homology \cite[Section 28.3]{KM}. Denoting by $d$ the absolute $\mathbb{Q}$-grading of the bottom of the $U$-tower of $\HMt_{\bullet}(Y,\spin)$, the Fr\o yshov invariant is defined to be $h(Y,\spin)=-d/2$ \cite[Chapter 39]{KM}. This is an important quantity that provides information about the negative definite manifolds bounded by $(Y,\spin)$. We have the following.
\begin{prop}\label{froybound}
The Fr\o yshov invariant $h(Y,\spin)$ satisfies the inequality
\begin{equation*}
|h(Y,\spin)|\leq \frac{1}{2}\left(|gr^{\mathbb{Q}}([B_{\bullet}])|+\max|\mathsf{sf}(\EHess\mathcal{L}_{(B_{\bullet},0)},\EHess_{\mathcal{L}(B,\Psi)})|+1\right).
\end{equation*}
\end{prop}
\begin{proof}
For every irreducible critical point $x=[B,\Psi]$, we have
\begin{equation*}
\mathrm{gr}^{\mathbb{Q}}(x)=\mathrm{gr}^{\mathbb{Q}}([B_{\bullet}])+\mathrm{gr}(x,[B_{\bullet}])=\mathrm{gr}^{\mathbb{Q}}([B_{\bullet}])+\mathsf{sf}(\EHess\mathcal{L}_{(B_{\bullet},0)},\EHess_{\mathcal{L}(B,\Psi)}).
\end{equation*}
Hence, the stable reducible critical point $[\mathfrak{b}_i]$ with minimum degree larger than
\begin{equation*}
M=\mathrm{gr}^{\mathbb{Q}}([B_{\bullet}])+\max\mathsf{sf}(\EHess\mathcal{L}_{(B_{\bullet},0)},\EHess_{\mathcal{L}(B,\Psi)})
\end{equation*}
is not the differential of any irreducible element for degree reasons, and is a cycle as differentials from boundary stable to irreducible critical points involve two-step trajectories involving boundary unstable critical points, which are also not present for degree reasons. Hence $[\mathfrak{b}_i]$ determines a non-zero class in the $U$-tower. Finally, there are no generators in the Floer chain complex in degrees less than
\begin{equation*}
m=\mathrm{gr}^{\mathbb{Q}}([B_{\bullet}])+\min\mathsf{sf}(\EHess\mathcal{L}_{(B_{\bullet},0)},\EHess_{\mathcal{L}(B,\Psi)})
\end{equation*}
and the result is proved.
\end{proof}

\vspace{0.5cm}
\section{Obstructions from $\Pin$-monopole Floer homology}\label{Pin2review}

We discuss our key tool to relate geometric spectral flow estimates to homology cobordism obstructions, namely $\Pin$-monopole Floer homology. This was introduced in \cite{Lin} as a Morse-Bott theoretic analogue of Manolescu's Floer-homotopic construction \cite{Man}. Let us briefly review the definitions and formal properties relevant for our purposes.
\\
\par
To each self-conjugate spin$^c$ structure $\spin$, i.e. $\spin=\bar{\spin}$, on a closed oriented three-manifold $Y$ we associate in \cite{Lin} the $\Pin$-\textit{monopole Floer homology groups}, which fit in the long exact sequence
\begin{equation}\label{pin2}
\cdots\stackrel{i_*}{\longrightarrow}\HSt_{*}(Y,\spin)\stackrel{j_*}{\longrightarrow} \HSf_{*}(Y,\spin)\stackrel{p_*}{\longrightarrow} \HSb_{*}(Y,\spin)\stackrel{i_*}{\longrightarrow}\cdots
\end{equation}
and are called respectively \textit{HS-to}, \textit{HS-from} and \textit{HS-bar}. These groups carry an absolute $\mathbb{Q}$-grading, and are graded modules over the ring
\begin{equation*}
\Rin= \ztwo[V,Q]/(Q^3)
\end{equation*}
where the actions of $V$ and $Q$ have degree respectively $-4$ and $-1$.
\par
We will be interested in the case if which $b_1(Y)=0$; in this case a self-conjugate spin$^c$ structure is the same as a spin structure. For any rational number $d$ let $\V_d$ and  $\V^+_d$ be the graded $\ztwo [V]$-modules $\ztwo[V^{-1},V]$ (the ring of Laurent power series) and $\ztwo[V^{-1},V]/V\ztwo [V]$ where the grading is shifted so that the element $1$ has degree $d$. We have the identifications as absolutely graded $\Rin$-modules up to an overall shift:
\begin{equation*}
\HSb_{*}(Y)\cong \V_2\oplus \V_1\oplus \V_0.
\end{equation*}
where action of $Q$ is an isomorphism from the first tower onto the second tower and from the second tower onto the third (and zero otherwise). The group $\HSt_{*}(Y,\spin)$ vanishes in degrees low enough, and the map $i_*$ is an isomorphism in degrees high enough. Hence $i_*\left(\HSb_{*}(Y,\spin)\right)$, considered as an $\ztwo[V]$-module, decomposes as the direct sum
\begin{equation*}
\V^+_c\oplus \V^+_b\oplus \V^+_a.
\end{equation*}
We call these three summands respectively the $\gamma$, $\beta$ and $\alpha$-towers. The action of $Q$ sends the $\gamma$-tower onto the $\beta$-tower and the $\beta$-tower onto the $\alpha$-tower. Manolescu's correction terms are then defined to be numbers
\begin{equation*}
\alpha(Y)\geq\beta(Y)\geq\gamma(Y)
\end{equation*}
such that
\begin{equation*}
a=2\alpha(Y),\quad b= 2\beta(Y)+1, \quad c=2\gamma(Y)+2.
\end{equation*}
The inequalities between these quantities follow from the module structure and in particular the $Q$-action between the towers. The fundamental properties of these numerical invariants are the following:
\begin{enumerate}
\item they are invariant under spin (and in particular $\mathbb{Z}/2\mathbb{Z}$-)homology cobordism.
\item they reduce modulo $2$ to $-\mu(Y,\spin)$, the Rokhlin invariant;
\item under orientation reversal we have
\begin{align*}
\alpha(-Y,\spin)&=-\gamma(Y,\spin)\\
\beta(-Y,\spin)&=-\beta(Y,\spin)\\
\gamma(-Y,\spin)&=-\alpha(Y,\spin).
\end{align*}
\end{enumerate}
\begin{remark}Indeed, the existence of a map $\beta:\Theta\rightarrow\mathbb{Z}$ with these properties directly implies that the Rokhlin homomorphism $\mu:\Theta\rightarrow\mathbb{Z}/2\mathbb{Z}$ does not split, hence that the Triangulation conjecture is false.
\end{remark}

We will be interested in the following basic invariant.
\begin{defn}
The $\Pin$-\textit{width} of $(Y,\spin)$ is the quantity $w(Y,\spin)=\alpha(Y,\spin)-\gamma(Y,\spin)$.
\end{defn}
This quantity is in fact a non-negative even number because $\alpha$ and $\gamma$ have the same reduction modulo $2$; it is manifestly a homology cobordism invariant. For example, one can show that
\begin{equation}\label{widthexample}
w(\Sigma(2,8n-1,16n-1))=2n.
\end{equation}
either using the description of $ \Sigma(2,8n-1,16n-1)$ as $(-1)$-surgery on the alternating torus knot $T(2,8n-1)$ \cite{Lin4} or more directly using the Seifert fibration \cite{Sto}.
\\
\par
It is in general very challenging to explicitly compute the Manolescu correction terms of a given $(Y,\spin)$, or even to understand its behavior under connected sums \cite{Lin2}, \cite{Lin4}, \cite{Sto2}. Fortunately, in order to prove our main result we will need a much less refined understanding of these invariants. We begin by showing how to relate the $\Pin$-width to the $U$-torsion. In particular, we show the following.
\begin{prop}\label{widthestimate}
The inequality $w(Y,\spin)\leq 4t(Y,\spin)+4$ holds for all $(Y,\spin)$.
\end{prop}

The key tool is the Gysin exact triangle relating $\Pin$-monopole Floer homology to its classical counterpart. In particular, we have for the \textit{to}-versions an exact triangle of the form
\begin{equation*}
\cdots \stackrel{\pi_*}{\longrightarrow} \HSt_{*}(Y,\spin)\stackrel{\cdot Q}{\longrightarrow}\HSt_{*}(Y,v)\stackrel{\iota_*}{\longrightarrow} \HMt_{*}(Y,\spin)\stackrel{\pi_*}{\longrightarrow} \HSt_{*}(Y,\spin)\stackrel{\cdot Q}{\longrightarrow} \cdots,
\end{equation*}
where the maps $\iota_*$ and $\pi_*$ have degree zero, see Chapter $4$ of \cite{Lin}. This is an exact sequence of $\Rin$-modules where on $\HMt_{*}$ we have that $V$ acts as $U^2$ and $Q$ acts as zero. There are analogue triangles for the \textit{from} and \textit{bar} versions; the latter looks like
\begin{center}
\begin{tikzpicture}
\matrix (m) [matrix of math nodes,row sep=1.5em,column sep=1em,minimum width=2em]
  {\cdot &\cdot &\cdot\\
  \ztwo &\ztwo& \ztwo\\
  \ztwo&0&\ztwo\\
  \ztwo &\ztwo&\ztwo\\};
  \path[-stealth]
  (m-2-1) edge node [above]{} (m-2-2)
      (m-4-2) edge node [above]{} (m-4-3)
       (m-2-3) edge node [above]{} (m-3-1)
        (m-3-3) edge node [above]{} (m-4-1)
  ;
\end{tikzpicture}
\end{center}
repeated in a $4$-periodic fashion. Here the middle column represents $\HMb_*$. In particular, in degrees high enough the Gysin triangle for the \textit{to} groups agrees with this too. 
\begin{proof}
Recall that the degree of element at the bottom of the tower in standard Floer homology $\HMt_*$ is denoted by $2\delta$ (so that $\delta=-h$ is the Fr\o yshov invariant with the opposite sign). Inspecting the map $\pi_*$ in the \textit{bar} version of the Gysin triangle (the arrows on the right side in the diagram above), we obtain by commutativity of the diagrams that the bottom of the $\alpha$-tower in $\HSt_*$ (which is in the image under $i_*$ of the elements in the bottom left and right of the diagram above) lies above of the bottom of the $U$-tower in $\HMt_*$. Notice that this argument shows that $\alpha\geq \delta$; on the other hand, there is in general no estimate relating $\beta$ and $\delta$.
\par
Consider now the portion of the $\beta$-tower in degree strictly below $2\delta$, denote by $\mathbf{x}$ the top degree element among them. Because the $\alpha$ tower ends in degrees $\geq 2\delta$, all the elements in this portion are killed by $Q$, hence are in the image of $\pi_*$. Choose $\mathbf{y}$ in $\HMt_*$ such that $\pi_*(\mathbf{y})=\mathbf{x}$; notice that $\mathbf{y}$ does not belong to the $U$-tower for degree (parity) reasons. If there are $k$ elements in this portion of the $\beta$-tower, $V^{k-1}\cdot \mathbf{x}\neq 0$ hence by commutativity of the diagram also 
\begin{equation*}
\pi_*(U^{2k-2}\cdot \mathbf{y})=V^{k-1}\cdot \mathbf{x}\neq 0,
\end{equation*}
so that $U^{2k-2}\cdot \mathbf{y}\neq 0$ and $t(Y,\spin)\geq 2k-1$. Hence
\begin{equation*}
\delta(Y,\spin)-\min\{\delta(Y,\spin),\beta(Y,\spin)\}\leq 2\cdot\#\{\text{elements in $\beta$ tower of degree $<2\delta$}\}=t(Y,\spin)+1.
\end{equation*}
Now, the Gysin triangle for the bar groups shows that the elements of the $\gamma$-tower are never in the image of the $U$-tower; the argument above can then be applied to the part of the $\gamma$-tower in degrees below $2\beta+1$ to obtain the bound
\begin{equation*}
\beta(Y,\spin)-\gamma(Y,\spin)\leq 2\cdot\#\{\text{elements in $\gamma$ tower of degree $<2\beta+1$}\}=t(Y,\spin)+1.
\end{equation*}
Putting the two estimates together, we obtain 
\begin{align*}
\delta(Y)-\gamma(Y)&=\delta(Y)-\beta(Y)+\beta(Y)-\gamma(Y)\leq\\
&\leq \delta(Y)-\min\{\delta(Y),\beta(Y)\}+\beta(Y)-\gamma(Y)\leq\\
&\leq2t(Y)+2.
\end{align*}
Finally, applying the inequality to $\bar{Y}$, we also obtain using $\delta(\bar{Y},\spin)=-\delta(Y,\spin)$ that
\begin{equation*}
\alpha(Y)-\delta(Y)\leq =2t({Y})+2,
\end{equation*}
(where we also used $t(\bar{Y},\spin)=t(Y,\spin)$ by Poincar\'e duality), and the result follows.
\end{proof}

We use this to prove the following result, which is the main tool we will use to study the homology cobordism group in the present paper.
\begin{prop}\label{obstrpin}
Suppose we have a collection of integral homology spheres $\{Z_k\}$ such that $t(Z_k)\leq N$ for all $C$. Then for every element $Y$ in the subgroup of the homology cobordism group generated by them, $w(Y)\leq 4N+4$. In particular, $\Sigma(2, 8n-1,16n-1)$ does not belong to that subgroup for $n>2N+2$.
\end{prop}

\begin{proof}
We have the following property
\begin{equation*}
t(Y\hash Y')=\mathrm{max}\{t(Y),t(Y')\},
\end{equation*}
which is a consequence of the isomorphism with the bigraded Tor
\begin{equation*}
\HMf_{*}(Y\hash Y')=\mathrm{Tor}^{\ztwo[[U]]}_{*,*}(\HMf_{*}(Y),\HMf_{*}(Y')).
\end{equation*} 
and the fact that, forgetting about the gradings,
\begin{align*}
\mathrm{Tor}^{\ztwo[U]}_{*,*}(\ztwo[U]/U^{n},\ztwo[U]/U^{m})&=\ztwo[U]/U^{n}\oplus \ztwo[U]/U^{n}\text{ for }n\leq m\\
\mathrm{Tor}^{\ztwo[U]}_{*,*}(\ztwo[U],\ztwo[U]/U^{n})&=\ztwo[U]/U^{n}.\\
\mathrm{Tor}^{\ztwo[U]}_{*,*}(\ztwo[U],\ztwo[U])&=\ztwo[U].
\end{align*}
This follows (with a little homological algebra over the principal ideal domain $\ztwo[U]$) from the description as mapping cone of Bloom, Mrowka and Ozsv\'ath (see \cite{Lin2}) or alternatively from the formula in Heegaard Floer homology \cite{OSz2} via the isomorphism between the theories (see \cite{HFHM1}, \cite{CGH} and subsequent papers). Because also $t(\bar{Y})=t(Y)$, we have that $t(X)\leq N$ for any $X$ obtained by connected sums of elements of our collection (with multiplicity and either orientation). The result then follows from Proposition \ref{widthestimate} and the computation (\ref{widthexample}).
\end{proof}

\begin{remark}
That the result readily generalizes to the case of $\mathbb{Z}/2\mathbb{Z}$-homology spheres.
\end{remark}

\begin{remark}One can easily show the stronger statement that for $n>2N+2$ the Brieskorn sphere $\Sigma(2, 8n-1,16n-1)$ does not belong to the subgroup generated by the $Z_k$s and all integral homology spheres $Y$ with $\alpha(Y)=\beta(Y)=\gamma(Y)$. This is because connecting sum with such a $Y$ does not change the value of the $\Pin$-width $w$ \cite{Lin2}. Furthermore, using instanton techniques \cite{Fur} one can show that the subgroup generated by these $Y$ contains a $\mathbb{Z}^{\infty}$. We decided to state the weaker version of the result as it is more closely relevant to our problem.
\end{remark}

\vspace{0.5cm}

\section{The Sobolev constant}\label{sobolev}
In this section we study the constant for the embedding $L^2_1\hookrightarrow L^6$; the content is essentially that of \cite[Section 1]{Li} with the dimensional constants made explicit and some added explanations for the reader's convenience. We will work with real valued functions for most of the section, in which case the natural $L^2_1$ norm we consider is
\begin{equation*}
\|f\|^2_{L^2_1}=\|f\|_{L^2}^2+\|df\|_{L^2}^2,
\end{equation*}
and then discuss the closed of coexact $1$-forms.
\\
\par
We begin by recalling the following definition from \cite{Li}, which is closely related to the Sobolev embedding  $L^1_1\hookrightarrow L^{3/2}$.

\begin{defn}
The Sobolev constant $\mathfrak{S}_Y$ of $M$ is the best constant for which the inequality
\begin{equation*}
\mathfrak{S}_Y\cdot \inf_{a\in\mathbb{R}}\|f-a\|^3_{L^{3/2}}\leq \|df\|^3_{L^1}
\end{equation*}
holds.
\end{defn}

Our goal is to relate $\mathfrak{S}_Y$ with the constant of the embedding $L^2_1\hookrightarrow L^6$. We begin with the following.

\begin{lemma}
If $f\in L^2_1$ and $\int_Y \mathrm{sign}(f)|f|^2d\mathrm{vol}=0$, then the inequality
\begin{equation*}
 \frac{\mathfrak{S}_Y^{2/3}}{16}\|f\|^2_{L^6}\leq \|df\|^2_2
\end{equation*}
holds
\end{lemma}
\begin{proof}
Set $g=\mathrm{sign}(f)|f|^4$, so that $\int_Y \mathrm{sign}(g)|g|^{1/2}d\mathrm{vol}=0$, hence 
\begin{equation*}
\|g\|_{L^{3/2}}=\inf_{a\in\mathbb{R}}\|g-a\|_{L^{3/2}}
\end{equation*}
(see the discussion in \cite[IV.5]{Cha}). By definition of the Sobolev constant we hence obtain
\begin{align*}
\mathfrak{S}_Y\|f\|^{12}_{L^6}&=\mathfrak{S}_Y\|g\|^3_{L^{3/2}}\leq \|dg\|^3_{L^1}\\
&=\|4f^3df\|^3_{L^1}\leq 64\|f\|_{L^6}^{9}\|df\|^3_{L^2}.
\end{align*}
where the last inequality follows from Cauchy-Schwartz. Hence
\begin{equation*}
\mathfrak{S}_Y\|f\|^{3}_{L^6}\leq 64 \|df\|^3_{L^2}
\end{equation*}
and the result follows.
\end{proof}

This is close to what we want, except of the assumption about the integral. It can be weakened via the following.
\begin{prop}For all $f\in L^2_1$, the inequality
\begin{equation*}
\|df\|^2_{L^2}\geq \frac{\mathfrak{S}_Y^{2/3}}{8}\left(\frac{1}{2^{8/3}}\|f\|_{L^6}^2- V^{-2/3}\|f\|^2_{L^2}\right)
\end{equation*}
holds
\end{prop}
 Notice that the right hand side can be negative (e.g. for constants).

\begin{proof}
Choose $k$ such that 
\begin{equation}\label{intzero}
\int \mathrm{sign}(f-k)|f-k|^2d\mathrm{vol}=0.
\end{equation}
By the previous lemma, this implies that
\begin{equation*}
\|df\|^2_{L^2}=\|d(f-k)\|^2_{L^2}\geq \frac{\mathfrak{S}_Y^{2/3}}{16}\|f-k\|^2_{L^6}
\end{equation*}
Now recall that $t\mapsto t^6$ is a convex function, hence
\begin{equation*}
\left(\frac{x+y}{2}\right)^6\leq \frac{x^6+y^6}{2}.
\end{equation*}
Hence
\begin{equation*}
\int f^6d\mathrm{vol}=\int ((f-k)+k)^6d\mathrm{vol}\leq 32(\int (f-k)^6d\mathrm{vol}+\int k^6d\mathrm{vol})=32(\int (f-k)^6d\mathrm{vol}+ k^6\cdot V).
\end{equation*}
so that
\begin{equation}\label{almostthere}
\|f-k\|^6_{L^6}=\int (f-k)^6d\mathrm{vol}\geq\frac{1}{32}\|f\|_{L^6}^6-k^6\cdot V.
\end{equation}
Notice that in this step we lose sharpness for constant functions, and the right hand side is possibly negative. Now if $a\geq b\geq0$ we have
\begin{equation*}
(a^3-b^3)\geq (a-b)^3
\end{equation*}
so that, if the right hand side of (\ref{almostthere}) is positive,
\begin{equation*}
\frac{1}{32}\|f\|_{L^6}^6-k^6\cdot V\geq (\frac{1}{2^{5/3}}\|f\|_{L^6}^2-k^2 V^{1/3})^3.
\end{equation*}
Hence putting everything together
\begin{equation*}
\|df\|^2_{L^2}\geq \frac{\mathfrak{S}_Y^{2/3}}{16}\left(\frac{1}{2^{5/3}}\|f\|_{L^6}^2-k^2 V^{1/3}\right),
\end{equation*}
where again the right hand side might be negative.
\par
We are left to understand the relation between $k$ and the $L^2$ norm of $f$, which is done as follows. Set $Y_{\pm}$ the set where $f-k$ is positive/negative respectively, and denote their volume by $V_{\pm}$. Then (\ref{intzero}) implies that
\begin{equation}\label{equalint}
\int_{Y_+}(f-k)^2d\mathrm{vol}=\int_{Y_-}(f-k)^2d\mathrm{vol}.
\end{equation}
As $f>k$ on $Y_+$,
\begin{equation*}
\int_{Y_+}(f-k)^2d\mathrm{vol}\leq \int_{Y_+}f^2d\mathrm{vol}-V_+\cdot k^2.
\end{equation*}
On $Y_-$, using the convexity estimate
\begin{equation*}
\frac{k^2}{2}\leq f^2+(f-k)^2
\end{equation*}
we obtain
\begin{equation*}
\int_{Y_-}(f-k)^2d\mathrm{vol}\geq -\int_{Y_-}f^2d\mathrm{vol}+V_-\cdot \frac{k^2}{2}.
\end{equation*}
Hence by (\ref{equalint}) we see
\begin{equation*}
\int_{Y_+}f^2d\mathrm{vol}-V_+\cdot k^2\geq -\int_{Y_-}f^2d\mathrm{vol}+V_-\cdot \frac{k^2}{2}
\end{equation*}
so that as $V\geq V_++V_-$ we have
\begin{equation*}
\|f\|_{L^2}^2\geq V\cdot\frac{k^2}{2}.
\end{equation*}
Plugging this into (\ref{almostthere}), we obtain the desired result.
\end{proof}

Now, every hyperbolic three manifold has volume $\geq0.94$ \cite{Mil}, so that
\begin{equation*}
\|df\|^2_{L^2}\geq \frac{\mathfrak{S}_Y^{2/3}}{8}\left(\frac{1}{2^{8/3}}\|f\|_{L^6}^2- 1.05\cdot \|f\|^2_{L^2}\right)
\end{equation*}
Rearranging the inequality, we obtain the following.
\begin{cor}\label{sobfun}
For every $f\in L^2_1$, the inequality
\begin{equation*}
\|f\|^2_{L^2_1}\geq \frac{\min(\mathfrak{S}_Y^{2/3}/8,1/1.05)}{2^{8/3}}\|f\|_{L^6}^2
\end{equation*}
holds.
\end{cor}
Hence we obtained control of the constant for the embedding $L^2_1\hookrightarrow L^6$ in terms of the Sobolev constant $\mathfrak{S}_Y$, in the case of functions. Notice that in general we will use an explicit lower bound for $\mathfrak{S}_Y^{2/3}/8$ which is much smaller than $1/1.05$, hence we can drop the latter.
\\
\par
Finally we come to the situation we are actually interested in, i.e. the case of $1$-forms $\alpha$. In this case, the natural norm we are interested in is
\begin{equation}\label{normforms}
\|\alpha\|^2_{L^2_1}=\|\alpha\|_{L^2}^2+\|\nabla\alpha\|_{L^2}^2.
\end{equation}
\begin{thm}
For every $1$-form $\alpha\in L^2_1$, the inequality
\begin{equation*}
\|\alpha\|^2_{L^2_1}\geq \frac{\min(\mathfrak{S}_Y^{2/3}/8,1/1.05)}{2^{8/3}}\|\alpha\|_{L^6}^2
\end{equation*}
holds.
\end{thm}
\begin{proof}
Denote by $C$ the constant appearing on the right hand side of Corollary \ref{sobfun}. The key observation is that by Kato inequality
\begin{equation}\label{kato}
|d|\alpha||\leq |\nabla \alpha|\text{ where }\alpha\neq0
\end{equation}
one obtains assuming $\alpha\neq 0$ everywhere that
\begin{equation*}
C\|\alpha\|^2_{L^6}=C\||\alpha|\|^2_{L^6}\leq \||\alpha|\|_{L^2}^2+\|d|\alpha|\|_{L^2}^2\leq  \|\alpha\|_{L^2}^2+\|\nabla\alpha\|_{L^2}^2=\|\alpha\|_{L^2_1}^2
\end{equation*}
where we used the inequality of Corollary \ref{sobfun} for the function $|\alpha|$. The argument easily generalizes to show that the inequality holds for example forms with non-degenerate zeroes, hence one can conclude by density as the latter are dense in the $C^{\infty}$ topology.
\end{proof}

\vspace{0.5cm}
\section{An explicit estimate in elliptic regularity}\label{explicitellip}

In our approach (outlined in Section \ref{SWreview}), we are interested in providing for an imaginary valued coexact $1$-form $b$ explicit estimates for its $L^6$ norm in terms of the $L^2$ norm of $db$. Such an estimate exists because the operator $d+d^*$ is elliptic: in particular, for every $1$-form $\alpha$ there is an estimate of the form
\begin{equation*}
\|\alpha\|_{L^2_1}\leq C(\|(d+d^*)\alpha\|_{L^2}+\|\alpha\|_{L^2}),
\end{equation*}
where the final term can be dropped when $d+d^*$ has no kernel, i.e. $b_1(Y)=0$. As in our case $d^*b=0$, this implies the inequality we are aiming at with an implicit constant because of the Sobolev embedding $L^2_1\hookrightarrow L^6$. In this present section, we will make the estimate explicit for a hyperbolic $Y$ in terms of its geometry.
\par
To start, in the case of a hyperbolic manifold the classical Bochner formula \cite{Bes} says that on $1$-forms
\begin{equation}\label{boch}
(d+d^*)^2=\nabla^*\nabla-2,
\end{equation}
where $-2$ is the Ricci curvature of a hyperbolic three-manifold. Integrating by parts we for a coexact $1$-form $b$ the identity
\begin{equation}\label{normcoexact}
\|b\|_{L^2_1}^2=3\|b\|_{L^2}^2+\|d b\|_{L^2}^2
\end{equation}
where we use the $L^2_1$-norm of $1$-forms defined in (\ref{normforms}).
\par
Consider now $\lambda_1^*$, the first eigenvalue of the Hodge Laplacian acting on coexact $1$-forms. This can be characterized in a variational way in terms of Rayleigh quotients as
\begin{equation*}
\lambda^*_1=\inf\frac{\|d b\|_{L^2}^2}{\|b\|_{L^2}^2},
\end{equation*}
the infimum being taken over coexact $1$-forms. In particular we have from (\ref{normcoexact}) that
\begin{equation*}
\|b\|_{L^2_1}^2\leq \left(1+\frac{3}{\lambda_1^*}\right) \|db\|_{L^2}^2.
\end{equation*}
This inequality is of course sharp in the case of a $\lambda_1^*$-eigenform.
\begin{remark}
When acting on coexact $1$-forms, the Hodge Laplacian is the square of the self-adjoint first order operator $\ast d$ which is relevant to our story; the latter has spectrum unbounded in both direction and $\lambda^*_1$ is the square of the eigenvalue of $\ast d$ with least absolute value.
\end{remark}
The results of the previous sections then give us bounds on the $L^6$ norm in terms of the Sobolev constant $\mathfrak{S}_Y$ of $Y$. Of course, the latter is still defined in analytical terms, but it turns out that one can provide explicit lower bounds in terms of the volume and injectivity radius of $Y$ as follows.
\par
First of all, recall that the inequality $\mathfrak{S}_Y\geq \mathfrak{I}_Y$ holds, where $\mathfrak{I}_Y$ is the \textit{isoperimetric constant} of $M$ \cite[Section IV.5]{Cha}: this is the infimum of the quantity
\begin{equation*}
\frac{A(N)^3}{\min(V(M_1),V(M_2))^2} 
\end{equation*}
over all hypersurfaces $N$ diving $Y$ in two pieces $M_1$, $M_2$.
\begin{remark}
One should not confuse $\mathfrak{I}_Y$ with the \textit{Cheeger constant} of $M$, which is the infimum of the quantity
\begin{equation*}
\frac{A(N)}{\min(V(M_1),V(M_2))} 
\end{equation*}
and closely related to the first eigenvalue of the Laplacian on functions on $M$.
\end{remark}
The isoperimetric constant $\mathfrak{I}_Y$ is an easier quantity to study as it can be bounded more directly using comparison techniques in Riemannian geometry. In particular, the following result holds.
\begin{thm}[\cite{Cha}, Chapter $V.3$]
Consider a closed hyperbolic three-manifold (or more generally any manifold with Ricci curvature $\geq-2$). Then 
\begin{equation*}
\mathfrak{I}_Y\geq \frac{1}{64\pi^5}\left(\frac{\mathrm{vol}(Y)}{\cosh(d(Y))-1}\right)^4
\end{equation*}
where $d(Y)$ denotes the diameter of $M$, i.e. the maximum Riemannian distance between two points on the manifold.
\end{thm}

In the case of a closed hyperbolic three-manifold, we can always assume $\mathrm{vol}(Y)\geq 0.94$ by \cite{Mil}, and obtain the following.
\begin{cor}
Consider a closed hyperbolic three-manifold . Then for the Sobolev constant $\mathfrak{S}_Y$ the inequality
\begin{equation*}
\mathfrak{S}_Y\geq \frac{4}{10^5}\frac{1}{(\cosh(d(Y))-1)^4}
\end{equation*}
holds.
\end{cor}

Finally, we can give an explicit upper bound on the diameter of a hyperbolic three-manifold in terms of volume and injectivity radius as follows.
\begin{thm}[\cite{Mey}]\label{diamineq}
For a closed hyperbolic three-manifold, the inequality
\begin{equation*}
d(Y)\leq \frac{\mathrm{vol(Y)}}{\pi\sinh^2(\mathrm{inj(Y)}/2)}
\end{equation*}
holds.
\end{thm}

This result is proved by explicitly studying the tube around the shortest geodesics connecting two points in $Y$ realizing the diameter of $Y$.
\\
\par
Putting everything together, we can finally prove the following explicit bound for the $L^6$ norm of a coexact $1$-form $b$ in terms of the $L^2$ norm of $db$.

\begin{prop}
Consider $V,\varepsilon,\delta>0$. Consider a hyperbolic rational homology sphere $Y$ with $\mathrm{vol}\leq V$, $\mathrm{inj}\geq\varepsilon$ and $\lambda_1^*\geq\delta$. Then for every coexact $1$-form $b$ on $Y$, the inequality
\begin{equation*}
\|b\|^2_{L^6}\leq e^{11}\cdot \left[\cosh\left(\frac{V}{\pi\sinh^2(\varepsilon/2)}\right)-1\right]^{8/3}\cdot\left(1+\frac{3}{\delta}\right) \|db\|_{L^2}^2.
\end{equation*}
holds.
\end{prop}

To conclude, for our purposes we need an estimate for the embedding constant $L^2_1\hookrightarrow L^4$. This is readily obtained from the one for $L^2_1\hookrightarrow L^6$ via the inequality
\begin{equation*}
\|b\|_{L^4}\leq \mathrm{vol}(Y)^{1/12}\cdot \|b\|_{L^6},
\end{equation*}
which is a direct consequence of H\"older's inequality. We obtain.
\begin{cor}\label{coexsob}
Consider $V,\varepsilon,\delta>0$. Consider a hyperbolic rational homology sphere $Y$ with $\mathrm{vol}\leq V$, $\mathrm{inj}\geq\varepsilon$ and $\lambda_1^*\geq\delta$. Then for every coexact $1$-form $b$ on $Y$, the inequality
\begin{equation*}
\|b\|_{L^4}\leq \mathfrak{c}_{V,\varepsilon}\cdot \left(1+\frac{3}{\delta}\right)^{1/2}\cdot \|db\|_{L^2}.
\end{equation*}
holds, where
\begin{equation*}
\mathfrak{c}_{V,\varepsilon}=e^{11/2}\cdot V^{1/12}\cdot \left[\cosh\left(\frac{V}{\pi\sinh^2(\varepsilon/2)}\right)-1\right]^{4/3}
\end{equation*}
is a constant depending only on $V$ and $\varepsilon$.
\end{cor}

While we have worked with $1$-forms so far, we will also need parts of this results for the case of spinors, in particular the constant of the embedding $L^2_1\hookrightarrow L^4$. In this case we consider the fixed base connection $B_{\bullet}$ with $B_{\bullet}^t$ flat, and the natural norm
\begin{equation*}
\|\Psi\|^2_{L^2_1}:=\|\Psi\|^2_{L^2}+\|\nabla_{B_{\bullet}}\Psi\|^2_{L^2}.
\end{equation*} 
All of our arguments (and in particular the Kato inequality (\ref{kato})) hold in this setup, and we obtain
\begin{cor}\label{diracsob}
Consider $V,\varepsilon,\delta>0$. Consider a hyperbolic rational homology sphere $Y$ with $\mathrm{vol}\leq V$, $\mathrm{inj}\geq\varepsilon$. Then for every section $\Psi$ of the spinor bundle $S\rightarrow Y$ the inequality
\begin{equation*}
\|\Psi\|_{L^4}\leq \mathfrak{c}_{V,\varepsilon}\cdot \|\Psi\|_{L^2_1}
\end{equation*}
holds.
\end{cor}

\vspace{0.5cm}
\section{Bounds on the spectrum under perturbation}\label{functional}

We now discuss some results about the spectral flow of a family of operators of the form $T+sA$ for $s\in[0,1]$. More specifically, we are interested in the case of the family of extended Hessians (\ref{exthess}) in which case
\begin{equation*}
T:=\EHess\mathcal{L}_{(B_{\bullet},0)}=\begin{bmatrix}
0 & -d^*&0 \\
-d&\ast d&0\\
0&0&D_{B_{\bullet}}
\end{bmatrix}
\end{equation*}
is a first order elliptic self-adjoint operator acting on $i\Omega^0\oplus i\Omega^1\oplus \Gamma(S)$, and $A$ is a symmetric perturbation term.
\par
We will denote by $H$ and $H_1$ the $L^2$ and $L^2_1$ completions of the underlying space respectively; notice that $T+sA: H\rightarrow H$ is an unbounded operator diagonalizable for every $s$, with discrete spectrum unbounded in both direction. It will convenient for our purposes to fix a suitable orthonormal basis $u_n$ of eigenvectors for $T$ in $H$ (with corresponding eigenvalue $t_n$) as follows. Notice that by the Hodge theorem $i\Omega^1=id\Omega^0\oplus id^*\Omega^2$, and the $2\times 2$ block in the upper left can be written as
\begin{equation*}
\begin{bmatrix}
0 & -d^*&0 \\
-d&0&0\\
0&0&\ast d
\end{bmatrix}
\end{equation*}
when acting on $i\Omega^0\oplus id\Omega^0\oplus id^*\Omega^2$. The operator $\ast d$ and $D_{B_{\bullet}}$ acting on coexact $1$-forms and spinors respectively are diagonalizable, and we can consider an orthonormal basis of eigenvectors for them.
\begin{remark}\label{diracR}
While $D_{B_{\bullet}}$ is a complex linear operator, we are interested in the spectral flow of $T+sA$ as a real family of operators hence we take a basis of eigenvectors over the reals.
\end{remark}
Now the block 
\begin{equation}\label{blockop}
\begin{bmatrix}
0 & -d^* \\
-d&0
\end{bmatrix}
\end{equation}
acting on $i\Omega^0\oplus id\Omega^0$ has an orthonormal basis consisting of eigenvectors of the form:
\begin{itemize}
\item $\begin{bmatrix}
0\\
1/\sqrt{\mathrm{vol}}
\end{bmatrix}$, with eigenvalue $0$;
\item $\frac{1}{\sqrt{2}}\begin{bmatrix}
\pm df/\sqrt{\lambda} \\
f
\end{bmatrix}$, where $f$ is a unit length $\lambda$-eigenfunction with $\lambda> 0$, corresponding to the eigenvalue $\pm\sqrt{\lambda}$.
\end{itemize}
Finally, our basis $\{u_n\}$ is obtained by taking the union of the bases for the three blocks we have discusses.
\\
\par
To understand the change in the spectrum, let us begin with the easier case in which $A$ is a symmetric bounded operator from $H$ to $H$; we will denote its norm by $\|A\|$. This is a quite standard situation and is treated for example in \cite{Kat} using the resolvent formalism. While not directly relevant for our applications, it will be instructive to discuss this case in detail.
\par
Recall that the \textit{resolvent set} of $T$ is defined to be the set of $z\in\mathbb{C}$ for which the resolvent operator
\begin{equation*}
R(T,z)=(T-z)^{-1}.
\end{equation*}
is a well defined bounded operator $H\rightarrow H$. In our situation, the resolvent set is exactly the complement of the eigenvalues. Indeed, very concretely, consider for $z\in\mathbb{C}$ the operator $T-z$, which acts on our basis as
\begin{equation*}
u_n\mapsto (t_n-z)u_n.
\end{equation*}
Assuming $z$ is not an eigenvalue of $T$, we have that $(T-z)^{-1}$ sends
\begin{equation*}
u_n\mapsto (t_n-z)^{-1}u_n.
\end{equation*}
hence is bounded with norm $\sup|t_n-z|^{-1}$.
\par
Now, if we choose an operator $B$ with norm
\begin{equation*}
\|B\|\cdot\sup|t_n-z|^{-1}<1,
\end{equation*}
the formal expression (also called second Neumann formula) for the resolvent of $T+B$ \cite[II.3]{Kat}
\begin{equation}\label{pertres}
R(T+B,z)=R(T,z)\cdot(1+B\cdot R(T,z))^{-1}
\end{equation}
makes sense as a bounded operator $H\rightarrow H$ because $\|B\cdot R(T,z)\|<1$, hence $z$ does not belong to $\mathrm{spec}(T+B)$ either. In particular, this means that if $z\in \mathrm{spec}(T+B)$, then there exists an eigenvalue $t_n$ of $T$ for which $|t_n-z|\leq \|B\|$.
\par
We can use this to see how the eigenvalues $T+sA$ evolve for $s\in[0,1]$. For example, if $t_*$ is an eigenvalue of multiplicity $m$ for $s_0$, and there are no eigenvalues in $(t_*-\varepsilon, t_*+\varepsilon)$, then there exactly $m$ eigenvalues of the operator at $s_0\pm \frac{\varepsilon}{2\|A\|}$ corresponding to the evolution of $t_*$ in the family. Furthermore, all of these $m$ eigenvalues are in the range $[t_*-\frac{\varepsilon}{2\|A\|}, t_*+\frac{\varepsilon}{2\|A\|}]$. This implies that if $t(s)$ is any continuous family of eigenvalues for $T+sA$, we have that $t(1)$ belongs to the interval $[t(0)-\|A\|,t(0)+\|A\|]$. Hence, eigenvalues $t_n$ with $|t_n|\leq \|A\|$ cannot contribute to the spectral flow of the family $\{T+sA\}$. We have proved.

\begin{prop}
In the case $A$ is bounded as an operator on $H$, the spectral flow of the family $\{T+sA\}$ for $s\in[0,1]$ is bounded by the number of eigenvalues with absolute value at most $\|A\|+1$.
\end{prop}

The situation in our case of interest is more delicate, because the perturbation $A$ is not bounded as an operator $H\rightarrow H$, but only as an operator on the dense subspace $H_1$. In our situation, $A:H_1\rightarrow H$ is a multiplication operator, and its boundedness follows from the embedding of $L^2_1$ in $L^6$ hence in $L^4$. We will denote its norm by $\|A\|_1$ for the remaining of the section.
\par
To analyze this situation, let us consider again the resolvent $(T-z)^{-1}$. We now need to consider its norm as an operator $H\rightarrow H_1$. We begin by analyzing the $H_1$ norms of the eigenvectors $u_n$. We have the following:
\begin{itemize}
\item Using the Bochner formula (\ref{boch}), if $u_n$ is an eigenvector of $\ast d$ we see that $
\|u_n\|_1=\sqrt{3+t_n^2}$.
\item Using the Weitzenb\"ock formula
\begin{equation*}
D_{B_{\bullet}}^2=\nabla_{B_{\bullet}}^*\nabla_{B_{\bullet}}-\frac{3}{2},
\end{equation*}
where we use that $B_{\bullet}^t$ is flat and $s=-6$, we see that if $u_n$ is an eigenvector of $D_{B_{\bullet}}$ then $\|u_n\|_1=\sqrt{5/2+t_n^2}$
\item if $u_n$ is an eigenvector of the block (\ref{blockop}) with non-zero eigenvalue, then $\|u_n\|_1=\sqrt{2+t_n^2}$ again by the Bochner formula. Of course, the eigenvector with eigenvalue $0$ has $\|u_n\|_1=1$.
\end{itemize}
In particular, for every element in our basis we have
\begin{equation*}
\|u_n\|_1\leq \sqrt{3+t_n^2}\leq 2+|t_n|.
\end{equation*}
A similar argument shows the $\{u_n\}$ are mutually orthogonal with respect to the natural inner product inducing the norm on $H_1$.

Using this, we see that again the spectrum of $T$ is exactly the complement of the set for which the resolvent $R(T,z)$ is bounded as an operator $H\rightarrow H_1$, and in this case its norm is bounded above by
\begin{equation*}
\sup\frac{2+|t_n|}{|t_n-z|}.
\end{equation*}
Using the formal expression for the resolvent of $T+B$ as above (\ref{pertres}), thought of now as an identity for operators $H\rightarrow H_1$, we conclude that if 
\begin{equation*}
\|B\|_1\cdot\sup\frac{2+|t_n|}{|t_n-z|}<1,
\end{equation*}
then $z$ does not belong to $\mathrm{spec}(T+B)$. Hence, if $z\in \mathrm{spec}(T+B)$, then there exists an eigenvalue $t_n$ of $T$ for which
\begin{equation*}
\frac{|t_n-z|}{2+|t_n|}\leq \|B\|_1.
\end{equation*}
To obtain results about the spectral flow, assume now that $\|B\|_1\leq \varepsilon$, so that if $z\in \mathrm{spec}(T+B)$,
\begin{equation*}
|t_n-z|\leq  2\varepsilon+\varepsilon{|t_n|}
\end{equation*}
for some eigenvalue $t_n$. In particular,
\begin{equation}\label{eigineq}
-2\varepsilon+(1-\varepsilon)t_n\leq z\leq 2\varepsilon+(1+\varepsilon)t_n
\end{equation}
for some eigenvalue $t_n$ of $T$. While more complicated than the case of bounded perturbations, this is still enough to conclude the following.
\begin{prop}\label{abstractbound}
In the case $A:H_1\rightarrow H$ is bounded, the spectral flow of the family $\{T+sA\}$ for $s\in[0,1]$ is bounded by the number of eigenvalues of $T$ with absolute value at most $2e^{\|A\|_1}$.
\end{prop}

\begin{proof}We need to show that eigenvalues larger in norm than $2e^{\|A\|_1}$ do not contribute to the spectral flow. We focus on the case of positive eigenvalues, as the case of negative ones is similar. To see this, fix $N$ large and consider the sequence of operators $T+\frac{k}{N}A$ for $k=0,\dots, N$; we think of it as a sequence of perturbations of norm $\varepsilon=\|A\|_1/N$. If $t_*>0$ is an eigenvalue of $T+\frac{k}{N}A$, then for $N$ large enough (\ref{eigineq}) implies that the corresponding eigenvalue of $T+\frac{k+1}{N}A$ (in a continuous family) is at least $t_*(1-\varepsilon)-2\varepsilon$. Hence we can provide an upper bound on the eigenvalues possibly contributing to spectral flow as follows. Consider the recurrence relation
\begin{align*}
x_0&=0\\
x_{n+1}&=\frac{x_n+2\varepsilon}{1-\varepsilon}.
\end{align*}
Then if $t_*>x_N$, it cannot contribute to the spectral flow. It is easy to see that
\begin{equation*}
x_k=2\left(\frac{1}{(1-\varepsilon)^k}-1\right),
\end{equation*}
so that
\begin{equation*}
x_N=2\left(\frac{1}{(1-\varepsilon)^N}-1\right)=2\left(\frac{1}{(1-\|A\|_1/N)^N}-1\right).
\end{equation*}
Letting $N$ to infinity, we see that if $t_*\geq2(e^{\|A\|_1}-1)$, $t_*$ cannot contribute to the spectral flow, and the result is proved.
\end{proof}

In the next section, we will only consider the operator norm $L^2_1\rightarrow L^2$, and drop the index from $\|A\|_1$.

\vspace{0.5cm}

\section{Geometric bounds on the perturbation}\label{spectralbound}
We now provide explicit bounds on the perturbation terms in family of extended Hessians of the Chern-Simons-Dirac functional discussed in Section \ref{SWreview}.
\par
Recall that for a bounded linear map $A:H\rightarrow H'$, the norm is defined as the least constant $\|A\|$ for which
\begin{equation*}
\|Ax\|\leq \|A\|\|x\|\text{ for all }x.
\end{equation*}
In our situation, we consider a map in block shape
\begin{equation*}
A=\begin{bmatrix}
0 & 0 &i\mathrm{Re}\langle i\Psi_0,\cdot\rangle \\
0&0 & \rho^{-1}(\cdot\Psi_0^*+\Psi_0\cdot^*)_0\\
\cdot \Psi_0& \rho(\cdot)\Psi_0 & \rho(b_0)\cdot
\end{bmatrix}=
\begin{bmatrix}
0&0&B\\
0&0&C\\
D&E&F
\end{bmatrix},
\end{equation*}
where the domain and codomain are $L^2_1$ and $L^2$ configurations respectively. Let us determine a bound on the norm of $A$ in terms of the blocks.
\begin{lemma}\label{block}
In the situation above, we have
\begin{equation*}
\|A\|\leq 3\cdot\max(\|B\|,\|C\|,\|D\|,\|E\|,\|F\|)
\end{equation*}
where we consider the norms $L^2_1\rightarrow L^2$.
\end{lemma}
\begin{proof}
We have
\begin{align*}
\|A(x_1,x_2,x_3)\|^2&=\|(Bx_3,Cx_3,Dx_1+Ex_2+Fx_3)\|^2=\\
&=\|Bx_3\|^2+\|Cx_3\|^2+\|Dx_1+Ex_2+Fx_3\|^2\\
&\leq 3(\|Bx_3\|^2+\|Cx_3\|^2+\|Dx_1\|^2+\|Ex_2\|^2+\|Fx_3\|^2)\\
&\leq 9\cdot\max(\|B\|,\|C\|,\|D\|,\|E\|,\|F\|)^2\cdot (\|x_1\|^2+\|x_2\|^2+\|x_3\|^2).
\end{align*}
where to go from the second line to the third line we used the Cauchy-Schwarz inequality.
\end{proof}

Let us then discuss the norms of these five operators.
\par
Notice that $B,C, D$ and $F$ depend on $\Psi_0$, for which we will have an $L^{\infty}$ bound. The computation of the norms of these operators in terms of this bound are all quite similar, and we focus on the slightly more complicated cases of $C$ and $E$. For $E$, let us first point out the pointwise identity
\begin{equation*}
|\rho(b)\Psi|=|b|\cdot|\Psi|.
\end{equation*}
This is readily checked because $\rho(b)$ is a hermitian matrix which after a unitary transformation we can assume to be
\begin{equation*}
\begin{bmatrix}
|b|&0\\
0&-|b|
\end{bmatrix}.
\end{equation*}
We can then bound the operator norm of $E$ as follows:
\begin{align*}
\|E({b})\|_{L^2}^2&=\int_Y |\rho({b})\Psi_0|^2d\mathrm{vol}=\int_Y|{b}|^2|\Psi_0|^2d\mathrm{vol} \leq (\max|\Psi_0|^2)\cdot \int_Y|{b}|^2d\mathrm{vol}\\
&=\max|\Psi_0|^2\cdot\|{b}\|_{L^2}^2\leq \max|\Psi_0|^2\cdot\|{b}\|_{L^2_1}^2
\end{align*}
so the norm of $E$ is bounded by $\max|\Psi_0|$.
\par
For $C$, via the Cauchy-Schwarz inequality we obtain the pointwise inequality
\begin{equation*}
|\rho^{-1}({\Psi}\Psi_0^*+\Psi_0{\Psi}^*)_0|^2\leq 2|\Psi_0|^2|{\Psi}|^2
\end{equation*}
so as in the case of $E$ we obtain that the norm of $E$ is bounded above by $\sqrt{2}\cdot\max|\Psi_0|$. Similarly, we also obtain that $\|B\|,\|D\|\leq \max|\Psi_0|$.
\par
The map $F$ is the problematic one because we will not have $L^{\infty}$ bounds on $b$. We can nevertheless bound the norm by
\begin{align*}
\|F({\Psi})\|_{L^2}^2&=\int_Y |\rho({b_0}){\Psi}|^2d\mathrm{vol}=\int_Y|{b_0}|^2|{\Psi}|^2d\mathrm{vol}\\
&\leq \left(\int_Y|{b_0}|^4d\mathrm{vol}\right)^{1/2}\left(\int_Y|{\Psi}|^4d\mathrm{vol}\right)^{1/2}\\
&=\|b_0\|_{L^4}^2\cdot \|{\Psi}\|_{L^4}^2\\
&\leq \mathfrak{c}_{V,\varepsilon}^2\cdot \|b_0\|_{L^4}^2\cdot \|{\Psi}\|_{L^2_1}^2
\end{align*}
where in the last step we used Corollary \ref{diracsob}. Hence $F$ has norm bounded above by $\mathfrak{c}_{V,\varepsilon}\cdot \|b_0\|_{L^4}$, and the norm of the operator $A$ is bounded by Lemma (\ref{block}) by
\begin{equation*}
\|A\|\leq 2\max\left\{\sqrt{2}\cdot\|\Psi_0\|_{L^{\infty}},\mathfrak{c}_{V,\varepsilon}\cdot\|b_0\|_{L^4}\right\}.
\end{equation*}
We have so far worked at a general configuration $(B_0,\Psi_0)$. If we now assume that this is a solution to the Seiberg-Witten equations (which we assume to be in Coulomb gauge with respect to $B_{\bullet}$, we have from (\ref{apriori}) the a priori estimate
\begin{equation*}
|\Psi_0|^2\leq 3
\end{equation*}
everywhere on $Y$, where we used that the scalar curvature of a hyperbolic $3$-manifold is $-6$. Now, by Corollary \ref{coexsob}, 
\begin{equation*}
\|b_0\|_{L^4}\leq \mathfrak{c}_{V,\varepsilon}\cdot \left(1+\frac{3}{\delta}\right)^{1/2}\cdot\|db_0\|_{L^2}
\end{equation*}
By the first Seiberg-Witten equation (\ref{SWfirst}) the pointwise identity
\begin{equation*}
|db_0|^2=|(\Psi_0\Psi_0^*)_0|^2=\frac{1}{4}|\Psi_0|^4
\end{equation*}
holds so that 
\begin{equation*}
\|db_0\|_{L^2}\leq \frac{3}{2}\cdot\mathrm{vol}(Y)^{1/2}
\end{equation*}
hence
\begin{equation*}
\|b_0\|_{L^4}\leq \frac{3}{2}\cdot\mathfrak{c}_{V,\varepsilon}\cdot \left(1+\frac{3}{\delta}\right)^{1/2}\cdot\mathrm{vol}(Y)^{1/2}.
\end{equation*}
Putting everything together, and noticing that the bound we have for $\|F\|$ is much worse than that of the other four operators, we obtain the following.
\begin{prop}\label{hessbound}
Consider $V,\varepsilon,\delta>0$. Consider a hyperbolic rational homology sphere $Y$ with $\mathrm{vol}\leq V$, $\mathrm{inj}\geq\varepsilon$ and $\lambda_1^*\geq\delta$. Then at any irreducible solution $(B_0,\Psi_0)$ to the Seiberg-Witten equations we have
\begin{equation*}
\EHess\mathcal{L}_{(B_0,\Psi_0)}=\EHess\mathcal{L}_{(B_{\bullet},0)}+A
\end{equation*}
with $A$ an operator with norm
\begin{equation*}
\|A\|\leq \frac{9}{2}\cdot \mathfrak{c}_{V,\varepsilon}\cdot V^{1/2}\cdot \left(1+\frac{3}{\delta}\right)^{1/2},
\end{equation*}
where we consider the norm as an operator $L^2_1\rightarrow L^2$ and the constant $\mathfrak{c}_{V,\varepsilon}$ is defined in Corollary \ref{coexsob}.
\end{prop}

\vspace{0.5cm}

\section{The spectral flow and conclusion of the proof}\label{mainproof}
From Proposition \ref{hessbound} and \ref{abstractbound}, we see that in order to bound the spectral flow in the family of extended Hessians we are left to bound the number eigenvalues of $\EHess\mathcal{L}_{(B_{\bullet},0)}$ in a given interval $[-T,T]$ for a given hyperbolic rational homology sphere. This can be done in several ways, and we will cite here the readily available explicit estimates (derived via Selberg trace formulas) in terms of injectivity radius and volume from \cite{LL2}. Alternative (weaker) bounds can be obtained under more general assumptions in terms of volume and diameter, see \cite{Li}.
\\
\par
We begin with the case of $\ast d$ and $D_{B_{\bullet}}$ acting on coexact $1$-forms and spinors respectively. For $\nu\geq 0$, denote by $\delta^*(\nu)$ and $\delta^D(\nu)$ the number of eigenvalues with with absolute value in $[\nu,\nu+1]$ respectively. In the case of the Dirac operator, we consider here its spectrum as a complex operator, cf. Remark \ref{diracR}. We then have the following estimate.
\begin{prop}[\cite{LL2}, Section 5]\label{localweyl}
For a given rational homology sphere $Y$ with $\mathrm{inj}\geq\varepsilon$ and $\mathrm{vol}\leq V$, there are upper bounds of the form
\begin{equation*}
\delta^*(\nu),\delta^D(\nu)\leq V\cdot(\mathfrak{a}_{\varepsilon}\cdot \nu^2+\mathfrak{b}_{\varepsilon})
\end{equation*}
where $\mathfrak{a}_{\varepsilon}$ and $\mathfrak{b}_{\varepsilon}$ are constants depending only of $\varepsilon$ which are \textbf{explicitly} computable. For example, for $\varepsilon=0.15$, we can take $\mathfrak{a}_{0.15}=3$ and $\mathfrak{b}_{0.15}=402$.
\end{prop}
More in general, the dependance of $\mathfrak{a}_{\varepsilon}$ and $\mathfrak{b}_{\varepsilon}$ on $\varepsilon$ only involves elementary operations but is somewhat unpleasant to write down explicitly. The proof follows directly from elementary formulas in Section $5$ of \cite{LL}, which provide explicit upper bounds in terms of volume and injectivity radius. In the case of coexact $1$-forms there is an extra constant term $C_{\varepsilon}$ (depending only on $\varepsilon$) which we convert into the desired form by 
\begin{equation*}
C_{\varepsilon}=\frac{C_{\varepsilon}}{\mathrm{vol(Y)}}\cdot\mathrm{vol(Y)}\leq \frac{C_{\varepsilon}}{0.94}\cdot\mathrm{vol(Y)},
\end{equation*}
where we used $\mathrm{vol}\geq0.94$.
\par
Given this, we can obtain the number of eigenvalues for either $\ast d$ or $D_{B_{\bullet}}$ with absolute value at most $T$ by summing the bounds over each interval $[i,i+1]$ for $i=0,\dots,\floor{T}$. Because
\begin{equation*}
0^2+1^2+2^2+\cdots (\floor{T}-1)^2=\frac{((\floor{T}-1)\floor{T}(2\floor{T}-1)}{6}\leq \frac{\floor{T}^3}{3}\leq \frac{{T}^3}{3}
\end{equation*}
and $\floor{T}+1\leq T^3$ for $T\geq 2$, we have the following (rougher) estimate.
\begin{cor}\label{weyl}
For a given rational homology sphere $Y$ with $\mathrm{inj}\geq\varepsilon$ and $\mathrm{vol}\leq V$, for both operators $\ast d$ and $D_{B_{\bullet}}$ the number of eigenvalues with absolute value in $[0,T]$ with $T\geq 2$ is bounded above by
\begin{equation*}
V\cdot(\frac{\mathfrak{a}_{\varepsilon}}{3}+\mathfrak{b}_{\varepsilon})\cdot T^3
\end{equation*}
where $\mathfrak{a}_{\varepsilon}$ and $\mathfrak{b}_{\varepsilon}$ are the same constants of Proposition \ref{localweyl}.
\end{cor}

The remaining part of the spectrum of $\EHess\mathcal{L}_{(B_{\bullet},0)}$ involves the block operator
\begin{equation}\label{block2}
\begin{bmatrix}
0 & -d^* \\
-d&0
\end{bmatrix}
\end{equation}
acting on $i\Omega^0\oplus id\Omega^0$, whose spectrum is closely related to that of the Laplacian $\Delta$ acting on functions. The latter is non-negative operator, and we will relabel the eigenvalues $\lambda_n$ in terms of the so-called spectral parameter $r_n\in i[0,1]\cup\mathbb{R}$ as $\lambda_n=1+r_n^2$. Eigenvalues less that $1$ (or, equivalently, corresponding to an imaginary parameter) are called \textit{small}. The relevance of this threshold is that $1$ is the bottom of the $L^2$ spectrum of the Laplacian on functions on $\mathbb{H}^3$. We will denote by $\delta^s$ the number of small eigenvalues, and $\delta(\nu)$ will denote the number of eigenvalues with parameter in $[\nu,\nu+1]$ for $\nu\geq 0$. We then have the following estimate.

\begin{prop}[\cite{LL2}, Section 5]\label{localweylfun}
For a given rational homology sphere $Y$ with $\mathrm{inj}\geq\varepsilon$ and $\mathrm{vol}\leq V$, there are upper bounds of the form
\begin{align*}
\delta^s&\leq \mathfrak{d}_{\varepsilon}\cdot V\\
\delta(\nu)&\leq V\cdot(\frac{\mathfrak{a}_{\varepsilon}}{2}\cdot \nu^2+\mathfrak{e}_{\varepsilon})
\end{align*}
where $\mathfrak{a}_{\varepsilon}$ is the constant from Proposition \ref{localweyl} and $\mathfrak{d}_{\varepsilon},\mathfrak{e}_{\varepsilon}$ are constants depending only of $\varepsilon$ which are also \textbf{explicitly} computable. For example, for $\varepsilon=0.15$, we can take $\mathfrak{d}_{0.15}=100$, $\mathfrak{e}_{0.15}=780$ .
\end{prop}

From here, we can obtain bounds on the spectrum of the block (\ref{block2}) as each eigenvalue $\lambda>0$ of $\Delta$ corresponds to two eigenvalues $\pm\sqrt{\lambda}$, while the eigenvalue $0$ of $\Delta$ contributes to a single eigenvalue to the block. We therefore have (again assuming $T\geq 2$) that
\begin{align*}
|\{\text{eigenvalues of block}\leq |T|\}|&\leq 2\cdot|\{\text{eigenvalues of Laplacian }\lambda=1+r^2\leq T^2\}|\\
&\leq 2\cdot \left(\delta^s+\sum_{i=0}^{\lfloor \sqrt{T^2-1}\rfloor}\delta(i)\right).
\end{align*}
The same estimates (performed in a much rougher way) as above give
\begin{equation*}
0^2+1^2+2^2+\cdots (\lfloor\sqrt{T^2-1}\rfloor)^2\leq (\lfloor\sqrt{T^2-1}\lfloor)^3 \leq T^3
\end{equation*}
and
\begin{equation*}
\lfloor\sqrt{T^2-1}\rfloor+1\leq T^3
\end{equation*}
so we obtain the following.

\begin{cor}\label{weylfun}
For a given rational homology sphere $Y$ with $\mathrm{inj}\geq\varepsilon$ and $\mathrm{vol}\leq V$, for the block (\ref{block2}) the number of eigenvalues with absolute value in $[0,T]$ with $T\geq 2$ is bounded above by
\begin{equation*}
2\cdot V\cdot\left[\mathfrak{d}_{\varepsilon}+(\frac{\mathfrak{a}_{\varepsilon}}{2}+\mathfrak{e}_{\varepsilon})\cdot T^3\right]\
\end{equation*}
where $\mathfrak{a}_{\varepsilon}$ and $\mathfrak{b}_{\varepsilon}$ are the same constants of Proposition \ref{localweylfun}.
\end{cor}

Putting everything together, and recalling that for spectral flow purposes we consider the spectrum of $D_{B_{\bullet}}$ as a real operator, we see that for $T\geq 2$ the number of eigenvalues of $\EHess\mathcal{L}_{(B_{\bullet},0)}$ with absolute value in $[0,T]$ with $T\geq 2$ is bounded above by
\begin{equation*}
V\cdot\left[2\mathfrak{d}_{\varepsilon}+(2\mathfrak{a}_{\varepsilon}+3\mathfrak{b}_{\varepsilon}+2\mathfrak{e}_{\varepsilon})\cdot T^3\right]
\end{equation*}
We can now apply Proposition \ref{abstractbound} to our setup to the family of extended Hessians. Combining Proposition \ref{hessbound} and with the estimate above, we obtain the following bound.
\begin{prop}\label{mainbound}
Consider $V,\varepsilon,\delta>0$. Consider a hyperbolic rational homology sphere $Y$ with $\mathrm{vol}\leq V$, $\mathrm{inj}\geq\varepsilon$ and $\lambda_1^*\geq\delta$. Then for any irreducible solution $(B_0,\Psi_0)$ to the Seiberg-Witten equations the absolute value of the spectral flow $|\mathsf{sf}(\EHess\mathcal{L}_{(B_{\bullet},0)},\EHess\mathcal{L}_{(B,\Psi)})|$ is bounded above by
\begin{equation*}
V\cdot\left[2\mathfrak{d}_{\varepsilon}+(2\mathfrak{a}_{\varepsilon}+3\mathfrak{b}_{\varepsilon}+2\mathfrak{e}_{\varepsilon})\cdot 8\cdot \exp\left(15\cdot \mathfrak{c}_{V,\varepsilon}\cdot V^{1/2}\cdot \left(1+\frac{3}{\delta}\right)^{1/2}\right)\right]
\end{equation*}
holds. Here the constants are those from Corollary \ref{coexsob} and Propositions \ref{localweyl} and \ref{localweylfun}.
\end{prop}

We can now complete the proofs of our main results; we simply need to put the pieces together.
\begin{proof}[Proof of Theorem \ref{mainthm}]We just need to put pieces together. The bound on the spectral flow of the Hessians in \ref{mainbound}, combined with Proposition \ref{torsbound} shows that if  for a hyperbolic rational homology sphere $Y$ with $\mathrm{vol}\leq V$, $\mathrm{inj}\geq\varepsilon$ and $\lambda_1^*\geq\delta$,
\begin{equation*}
t(Y)\leq 2\cdot V\cdot\left[2\mathfrak{d}_{\varepsilon}+(2\mathfrak{a}_{\varepsilon}+3\mathfrak{b}_{\varepsilon}+2\mathfrak{e}_{\varepsilon})\cdot 8\cdot \exp\left(15\cdot \mathfrak{c}_{V,\varepsilon}\cdot V^{1/2}\cdot \left(1+\frac{3}{\delta}\right)^{1/2}\right)\right]+2.
\end{equation*}
The conclusion then follows from Proposition \ref{obstrpin}, and we just need to plug in $\varepsilon=0.15$ in the values for the constants.
\end{proof}

\begin{proof}[Proof of Theorem \ref{secondthm}]
In \cite{LL2} it is shown that there exists an explicitly computable constant $\mathfrak{f}_{V,\varepsilon}$ for such that for all hyperbolic rational homology sphere $Y$ with $\mathrm{vol}\leq V$, $\mathrm{inj}\geq\varepsilon$, we have the inequality
\begin{equation*}
|\mathrm{gr}^{\mathbb{Q}}([\mathfrak{b}_0])|\leq \mathfrak{f}_{V,\varepsilon}
\end{equation*}
for the absolute grading of the first boundary stable critical point $[\mathfrak{b}_0]$. Notice that in \cite{LL} the assumption that $\lambda_1^*>2$ implies that the operator $D_{B_{\bullet}}$ has no kernel; in general, one must add to the bound from the eta invariants also the dimension of the kernel of $D_{B_{\bullet}}$, which can be bounded in terms of $V,\varepsilon$ via the local Weyl law \ref{localweyl}. Combining this with the bound for the spectral flow \ref{mainbound}, we conclude via Proposition \ref{froybound}.
\end{proof}

\vspace{0.5cm}

\bibliographystyle{alpha}
\bibliography{biblio}

\end{document}